\documentclass[12pt]{amsart}
\usepackage{amsmath, amssymb, amsthm, amsfonts, mathrsfs,latexsym}
\usepackage{enumerate}

\makeatletter
\@namedef{subjclassname@2020}{%
  \textup{2020} Mathematics Subject Classification}
\makeatother

\theoremstyle{plain}
\newtheorem{prop}{Proposition}[section]

\newtheorem{lemma}[prop]{Lemma}

\newtheorem{theorem}[prop]{Theorem}
\newtheorem*{maintheorem}{Main Theorem}

\usepackage{graphicx, epsfig}

\usepackage[colorlinks=true,linkcolor=blue,citecolor=red]{hyperref}
\usepackage{enumitem}

\newcommand{\B}{\ensuremath{\mathcal{B} } }
\newcommand{\R}{\ensuremath{\mathbb{R} } }

\newcommand{\N}{\ensuremath{\mathbb{N} } }
\newcommand{\hC}{\ensuremath{\mathcal{C} } }
\newcommand{\Ps}{\ensuremath{\mathcal{P}} }

\newcommand{\Id}{\ensuremath{\mathrm{Id}} }

\newcommand{\Aut}{\ensuremath{\mathrm{Aut}} }
\newcommand{\Mod}{\ensuremath{\mathrm{Mod}} }

\newcommand{\coloneqq}{\ensuremath{\colon\!\!\!\!=} }

\numberwithin{equation}{section}

\begin{document}


\baselineskip=17pt


\title[Borel Completeness of STS]{Borel completeness of the class of countable Steiner triple systems}

\author{Guangyin Ma}
\address{School of Mathematics\\Sichuan University\\24 South Section 1, First Ring Road\\Chengdu, Sichuan 610065\\
China}
\email{guangyinma@foxmail.com}

\author{Shichang Song}
\address{School of Mathematics and Statistics\\Beijing Jiaotong University\\3 Shangyuancun\\Haidian, Beijing 100044\\
China}
\email{ssong@bjtu.edu.cn}

\date{}

\begin{abstract}
We show that the isomorphism relation for countable Steiner triple systems is Borel complete, that is, the isomorphism relation for arbitrary countable structures is Borel reducible to that for countable Steiner triple systems. To prove it, we construct a faithful Borel reduction from countable graphs to countable Steiner triple systems, that is, a Borel assignment $\theta$ that associates every countable graph $G$ with a countable Steiner triple system $\theta(G)$ so that $G\cong G'$ if and only if $\theta(G)\cong\theta(G')$. Moreover, $\theta$ preserves automorphisms which means that $\Aut(G)\cong\Aut(\theta(G))$.
\end{abstract}

\subjclass[2020]{Primary: 03E15; Secondary: 05B07, 05C65}

\keywords{Borel completeness, Borel reducibility, isomorphism relations, 3-uniform hypergraphs, Steiner triple systems}

\maketitle

\section{Introduction}

A \emph{partial Steiner triple system} $(S,\B)$ is a set $S$ together with a collection $\B$ of subsets of $S$ of size 3 such that every pair of two distinct elements of $S$ belong to at most one element of $\B$. Elements of $S$ are called \emph{points}, and elements of $\B$ are called \emph{blocks}. A partial Steiner triple system is a \emph{Steiner triple system} if every pair of two distinct points belong to exactly one block. Recently, the study of countable Steiner triple systems has drawn increasing attention from both the model theory community and the design theory community. For instance, model theorists Barbina and Casanovas \cite{BC} investigated the model theory of Steiner triple systems, and design theorists Horsley and Webb \cite{HW} constructed uncountably many countable homogeneous Steiner triple systems.
In this paper, we will show that the classification problem for countable Steiner triple systems is equivalent to the classification problem for countable graphs, and thus the classification problem is as complex as possible among countable structures.

Borel complexity theory or invariant descriptive set theory provides a framework to compare complexities of different classification problems; Gao's book \cite{G} is a good reference.
In this theory, classification problems for countable structures are regarded as equivalence relations on standard Borel spaces. A \emph{standard Borel space} is a Polish space equipped just with the $\sigma$-algebra of its Borel sets. Every standard Borel space is isomorphic to a Borel subspace of $\R$.
Let $E$ and $F$ be equivalence relations on standard Borel spaces $X$ and $Y$ respectively. We say that $E$ is \emph{Borel reducible} to $F$, written as $E\leq_B F$, if there is a Borel function $f\colon X\to Y$ such that for all $x,x'\in X$,
$$x E x' \iff f(x) F f(x').$$
Such $f$ is called a \emph{Borel reduction} from $E$ to $F$.
Then, $\leq_B$ defines a partial preordering on equivalence relations on standard Borel spaces.
A Borel reduction $f\colon X\to Y$ from $E$ to $F$ is a \emph{faithful Borel reduction} if for every $E$-invariant Borel set $A\subseteq X$, the $F$-saturation of $f(A)$, $[f(A)]_F\coloneqq \{x\in Y \mid xFy \textrm{ for some } y\in f(A)\}$, is Borel.

Let $L=\{R_i\mid i\in I\}$ be a countable relational language, where $I$ is countable and each $R_i$ is an $n_i$-ary relation symbol. Let $\Mod(L)$ denote the space of all countable $L$-structures with universe $\N$. Then, $\Mod(L)$ is the set
$$\prod_{i\in I}\Ps(\N^{n_i}),$$
where $\Ps(\N^{n_i})$ is the power set of $\N^{n_i}$. This space is equipped with the product topology by identifying $\Ps(\N^{n_i})$ with $2^{\N^{n_i}}$ for each $i$. Then, $\Mod(L)$ is a standard Borel space.
Let $\sigma$ be an $L_{\omega_1,\omega}$-sentence, where countable conjunctions and disjunctions are allowed. Define
$$\Mod(\sigma)\coloneqq\{M\in\Mod(L)\mid M\models\sigma\}.$$
Then $\Mod(\sigma)$ is a Borel subset of $\Mod(L)$, and thus a standard Borel space by Kuratowski Theorem.
Also, $\Mod(\sigma)$ is the class of countable models of $\sigma$ with universe $\N$.
Let $\cong_\sigma$ denote the isomorphism relation on $\Mod(\sigma)$, which is the classification problem for countable models of $\sigma$. By the L\'{o}pez-Escobar theorem, a $\cong$-invariant subset $B\subseteq\Mod(L)$ is Borel if and only if there is an $L_{\omega_1,\omega}$-sentence $\tau$ such that $B=\Mod(\tau)$.

Friedman and Stanley \cite{FS} considered Borel reducibility between equivalence relations to compare isomorphism relations, and thus characterize complexities of different classification problems.
An equivalence relation $E$ is \emph{(faithfully) Borel complete} if for every countable relational language $L$ and every $L_{\omega_1,\omega}$-sentence $\sigma$, there is a (faithful) Borel reduction from $\cong_{\sigma}$ to $E$. Thus, a Borel complete equivalence relation is a $\leq_B$-maximal element among equivalence relations $\cong_\sigma$.
If $\cong_\sigma$ is (faithfully) Borel complete, we also say that the class of countable models of $\sigma$ is \emph{(faithfully) Borel complete}.
It is folklore that the class of countable graphs is faithfully Borel complete, which is the best-known example.
In Friedman and Stanley's seminal work \cite{FS}, they proved that the classes of countable trees, countable linear orderings, countable fields of characteristic $p$ ($p$ is a prime number or $p=0$) are all Borel complete, and many others. After Friedman and Stanley's work, more classes are shown to be Borel complete. For instance, Camerlo and Gao \cite{CG} proved that the class of countable boolean algebras is Borel complete, and Clemens, Coskey, and Potter \cite{CCP} proved that the class of countable vertex-transitive digraphs and partial orders are Borel complete.

More recently, a longstanding problem by Friedman and Stanley was solved, which is that the class of countable torsion free abelian groups is Borel complete. It was announced by Paolini and Shelah \cite{PS}, and later Laskowski and Ulrich \cite{LU} gave a new proof. Moreover, the class of countable torsion free abelian groups was shown to be faithfully Borel complete by Paolini and Shelah \cite{PS2}.

Let $\Gamma$ be the theory of graphs and let $\Sigma$ be the theory of Steiner triple systems. The main result of this paper is the following theorem.

\begin{maintheorem}\label{main}
The class of countable Steiner triple systems is faithfully Borel complete. Moreover, there is a faithful Borel reduction $\theta\colon\Mod(\Gamma)\to\Mod(\Sigma)$ such that for every countable graph $M$ with universe $\N$, we have that $\Aut(\theta(M))\cong\Aut(M)$.
\end{maintheorem}

Let $\Gamma_3$ be the theory of 3-uniform hypergraphs and let $\Sigma_0$ be the theory of partial Steiner triple systems. We will prove the main theorem by constructing three faithful Borel reductions $F\colon\Mod(\Gamma)\to\Mod(\Gamma_3)$, $G\colon\Mod(\Gamma_3)\to\Mod(\Sigma_0)$, and $H\colon\Mod(\Sigma_0)\to\Mod(\Sigma)$. Moreover, each reduction preserves automorphisms.


\section{Preliminaries}
In this section, we fix our notations. Also, we present definitions of graphs, hypergraphs, and (partial) Steiner triple systems.

\subsection{Notations}
Let $X$ be a set and let $n$ be a positive integer. Let $[X]^n$ denote the set of all subsets of $X$ of size $n$.
For two $n$-tuples $\bar{x}$ and $\bar{y}$, we write $\bar{x}\neq \bar{y}$ as shorthand for $\bigvee\limits_{i<n} (x_i\neq y_i)$,
and we write $ \{\bar{x}\}\neq \{\bar{y}\} $ as shorthand for
$$\bigl(\bigvee\limits_{i<n}\bigwedge\limits_{j<n} x_i\neq y_j\bigr) \vee\bigl(\bigvee\limits_{i<n}\bigwedge\limits_{j<n} y_i\neq x_j\bigr).$$
Let $\bar{x}$ be an $n$-tuple and let $\bar{y}$ be an $m$-tuple. Let $\varphi(\bar{x})$ be an $L_{\omega_1,\omega}$ formula with $n$ free variables.
We use abbreviations $\exists^{\neq}\bar{x}_0,\cdots,\bar{x}_{k-1}\varphi$ and $\forall^{\neq}\bar{x}_0,\cdots,\bar{x}_{k-1}\varphi$ for distinct tuples, and
$\exists^{\neq}\{\bar{x}_0\},\cdots,\{\bar{x}_{k-1}\}\varphi$ and $\forall^{\neq}\{\bar{x}_0\},\cdots,\{\bar{x}_{k-1}\}\varphi$ for distinct sets.
Then we are able to use the following notations to describe the exact number of tuples or sets satisfying an $L_{\omega_1,\omega}$ formula $\varphi(\cdot)$ with $n$ free variables:
\begin{align*}
\varphi(\cdot)^{(\cdot) \geqslant k} &\coloneqq \exists^{\neq}\bar{x}_0,\dots,\bar{x}_{k-1}\bigwedge\limits_{i<k}\varphi(\bar{x}_i),\\
\varphi(\cdot)^{\{\cdot\} \geqslant k} &\coloneqq \exists^{\neq}\{\bar{x}_0\},\dots,\{\bar{x}_{k-1}\}\bigwedge\limits_{i<k}\varphi(\bar{x}_i),\\
\varphi(\cdot)^{(\cdot) = k} &\coloneqq \neg\left( \varphi(\cdot)^{(\cdot) \geqslant k+1}\right) \wedge\varphi(\cdot)^{(\cdot) \geqslant k},\\
\varphi(\cdot)^{\{\cdot\} = k} &\coloneqq \neg\left( \varphi(\cdot)^{\{\cdot\} \geqslant k+1}\right)  \wedge \varphi(\cdot)^{\{\cdot\} \geqslant k},\\
\varphi(\cdot)^{(\cdot) = \omega} &\coloneqq \bigwedge\limits_{k\in\omega} \varphi(\cdot)^{(\cdot) \geqslant k},\\
\varphi(\cdot)^{\{\cdot\} = \omega} &\coloneqq \bigwedge\limits_{k\in\omega} \varphi(\cdot)^{\{\cdot\}\geqslant k}.
\end{align*}


\subsection{Graphs, hypergraphs and Steiner triple systems}
A \emph{graph} $G$ is a pair $G=(V,E)$, where $V$ is the set of elements called \emph{vertices}, and $E$ is a subset of $[V]^2$ called \emph{edges} such that for all $x,y\in V$, $(x,y)\in E$ if and only if $(y,x)\in E$.
Let $L^R=\left\lbrace R \right\rbrace$ denote the language of graphs, where $R$ is a binary relation symbol. The axiom $\Gamma$ for graphs is the conjunction of $\forall x\neg R(x,x)$ and $\forall x\forall y R(x,y)\leftrightarrow R(y,x)$.
A \emph{3-uniform hypergraph} is a pair $H=(V,E)$, where $V$ is the set of elements called \emph{vertices}, and $E$ is a subset of $[V]^3$ called \emph{edges} such that for all $x_1,x_2,x_3\in V$, if $(x_1,x_2,x_3)\in E$ then $(x_{\pi(1)},x_{\pi(2)},x_{\pi(3)})\in E$ for all permutations $\pi$ of $\{1,2,3\}$.
Let $L^E =\{E\}$ denote the language of hypergraphs, where $E$ is a ternary relation symbol. Then the axiom $\Gamma_3$ for 3-uniform hypergraphs is the conjunction of the following sentences:
$$ \forall x\,\forall y\,\forall z\;E(x,y,z)\rightarrow\left( x\neq y \wedge y\neq z\wedge z\neq x\right) , $$
$$ \forall x\,\forall y\,\forall z\;\left( E(x,y,z)\leftrightarrow E(y,x,z)\right) \wedge\left( E(x,y,z)\leftrightarrow E(x,z,y)\right).$$
A \emph{partial Steiner triple system} is a pair $(S,\B)$, where $S$ is the set of elements called \emph{points}, and $\B$ is a subset of $[S]^3$ called \emph{blocks} such that every two distinct points belong to at most one block. A partial Steiner triple system is a \emph{Steiner triple system} if every two distinct points belong to exactly one block.
The book \cite{CR} is a good reference for Steiner triple systems.
A (partial) Steiner triple system can be viewed as a 3-uniform hypergraph, where points are vertices and blocks are edges. Let $\Sigma_0$ denote the axiom for partial Steiner triple systems. Then $\Sigma_0$ is the conjunction of $\Gamma_3$, the axiom for 3-uniform hypergraphs, and the following $L^E$-sentence:
$$\forall x\,\forall y\,\forall z_1\,\forall z_2\;\left( E(x,y,z_1)\wedge E(x,y,z_2)\right) \rightarrow \left( z_1 = z_2 \right).$$
Let $\Sigma$ denote the axiom for Steiner triple systems, which is the conjunction of $\Sigma_0$ and the following $L^E$-sentence:
$$\forall x \forall y \left( x\neq y \right)\rightarrow\left(\exists z E(x,y,z)\right).$$

Naturally, for every Steiner triple system $(S,\B)$, it induces a multiplication $*$ on $S$ as follows: (i) for all $x\in S$, $x*x\coloneqq x$, (ii) for all distinct $x,y\in S$, $x*y\coloneqq z$, where $z\in S$ is the unique element such that $\{x,y,z\}\in\B$. Then $(S,*)$ becomes a \emph{Steiner quasigroup} satisfying: (i) $a*b=b*a$, (ii) $a*a=a$, (iii) $a*(a*b)=b$ for all $a,b\in S$. Clearly, every Steiner quasigroup $(S,*)$ also naturally induces a Steiner triple system $(S,\B)$. So, Steiner triple systems and Steiner quasigroups are essentially the same, and we view them as the same objects in this paper.
For a partial Steiner triple system $(S,\B)$, its induced multiplication $*$ is only a partial function, which means that if $a,b\in S$ are not in a block, then $a*b$ is not defined. Like the construction of free groups, a partial Steiner triple system $(S,\B)$ together with a partial multiplication $*$ can freely generate a quasigroup $\langle S\rangle$. Since every Steiner quasigroup as in its definition satisfies three relations, we quotient out by these three relations to get a Steiner quasigroup $\hat S\coloneqq\langle S\rangle/\sim$. We call $\hat S$ \emph{the freely generated Steiner quasigroup by the partial Steiner triple system $S$}.

Since this free generation construction is very important in our proof of the main theorem, we give more details here. Given a partial Steiner triple system $(S,\B)$, we let $S_0\coloneqq S$ and $\B_0\coloneqq\B$. For every two distinct $a,b\in S$, if there is $c\in S$ such that $\{a,b,c\}\in\B$, then we let $\mu(\{a,b\})\coloneqq c$, otherwise, we let $\mu(\{a,b\})$ be a new element. Let $S_1\coloneqq S_0\cup \{\mu(\{a,b\})\mid a\neq b\in S\}$ and $\B_1\coloneqq\B_0\cup \{\{a,b,\mu(\{a,b\})\}\mid a\neq b\in S_0\}$. Then $(S_1,\B_1)$ is still a partial Steiner triple system. Recursively, for each $n\in\N$, we define the partial Steiner triple system $(S_n,\B_n)$. Let $\hat S\coloneqq\bigcup_{n\in\N}S_n$ and $\hat \B\coloneqq\bigcup_{n\in\N}\B_n$. Then $(\hat S,\hat \B)$ is a Steiner triple system, and is called \emph{the freely generated Steiner triple system by the partial Steiner triple system $S$}. For each $n\in\in\N$, we call $S_n$ \emph{Level $n$} in the free generation construction. 

The following lemma is used in the proof of the main theorem.

\begin{lemma}\label{free generation preserves Aut}
Let $S$ be a partial Steiner triple system and let $\hat S$ be the freely generated Steiner triple system by $S$. Suppose that every automorphism of $\hat S$ fixes $S$ setwise. Then $\Aut(S)\cong\Aut(\hat S)$.
\end{lemma}
\begin{proof}
  For each $n\in\N$, let $S_n$ be Level $n$ in the free generation construction. Let $\sigma$ be an automorphism of $S$. Naturally, $\sigma$ can be extended to an automorphism $\sigma_n$ of $S_n$, for each $n\in\N$. Let $\hat\sigma\coloneqq \bigcup_{n\in\N}\sigma_n$. Then $\hat\sigma$ is an automorphism of $\hat S$. Define $f(\sigma)\coloneqq \hat\sigma$. It is easy to verify that (i) $f(\sigma_1\sigma_2)=f(\sigma_1)f(\sigma_2)$ for $\sigma_1,\sigma_2\in\Aut(S)$, and (ii) $f(\mathrm{Id}_S)=\mathrm{Id}_{\hat S}$. Thus, $f\colon\Aut(S)\to\Aut(\hat S)$ is a homomorphism. Let $\tau$ be an automorphism of $\hat S$. By assumption, $\tau(S)=S$, and thus $\tau|_S\in \Aut(S)$. Define $g(\tau)\coloneqq \tau|_S$. It is easy to verify that (i) $g(\tau_1\tau_2)=g(\tau_1)g(\tau_2)$ for $\tau_1,\tau_2\in\Aut(\hat S)$, and (ii) $g(\mathrm{Id}_{\hat S})=\mathrm{Id}_S$. Thus, $g\colon\Aut(\hat S)\to\Aut(S)$ is a homomorphism. Clearly, $g\circ f=\mathrm{Id}$ and $f\circ g=\mathrm{Id}$, and thus $\Aut(S)\cong\Aut(\hat S)$.  
\end{proof}

\section{The class of countable 3-uniform hypergraphs is faithfully Borel complete}
In this section, we will prove that the class of 3-uniform hypergraphs is faithfully Borel complete.
Let $\Gamma$ be the theory of graphs and let $\Gamma_3$ be the theory of 3-uniform hypergraphs.

\begin{prop}\label{hg}
The class of countable 3-uniform hypergraphs is faithfully Borel complete.
Moreover, there is a faithful Borel reduction $F\colon\Mod(\Gamma)\to\Mod(\Gamma_3)$ such that for every countable graph $M$ with universe $\N$, we have $\Aut(F(M))\cong\Aut(M)$.
\end{prop}
\begin{proof}
	It is well-known that the class of countable graphs is faithfully Borel complete.
It suffices to construct a faithful Borel reduction $F\colon \mathrm{Mod}(\Gamma)\to\mathrm{Mod}(\Gamma_3)$.
Let $M\in\mathrm{Mod}(\Gamma)$ and we define $F(M)$ as follows:
	\begin{enumerate}[label=(\itshape\arabic*)]
		\item Based on the graph $M=(\N,R)$, we construct a 3-uniform hypergraph $(F_0(M),E)$ on $\{-4,-3,-2,-1\}\cup\N$, where edges are defined as follows:
			\begin{itemize}
				\item $E(-1,-3,-4)$
				\item $E(-1,m,n)$ for all $m,n\in\N$ with $R(m,n)$
				\item $E(-1,-2,m)$ for every $m\in \N\cup\{-4\}$
			\end{itemize}
        See Figure~\ref{fig:3uh} for the construction.
		\item Define $c\colon \{-4,-3,-2,-1\}\cup\N\to\N$ by $c(m) = m+4$. We get a 3-uniform hypergraph $F(M)$ with universe $\N$ from $F_0(M)$.
	\end{enumerate}
	\begin{figure}[!htbp]
		\centering
		\includegraphics[width=0.80\textwidth]{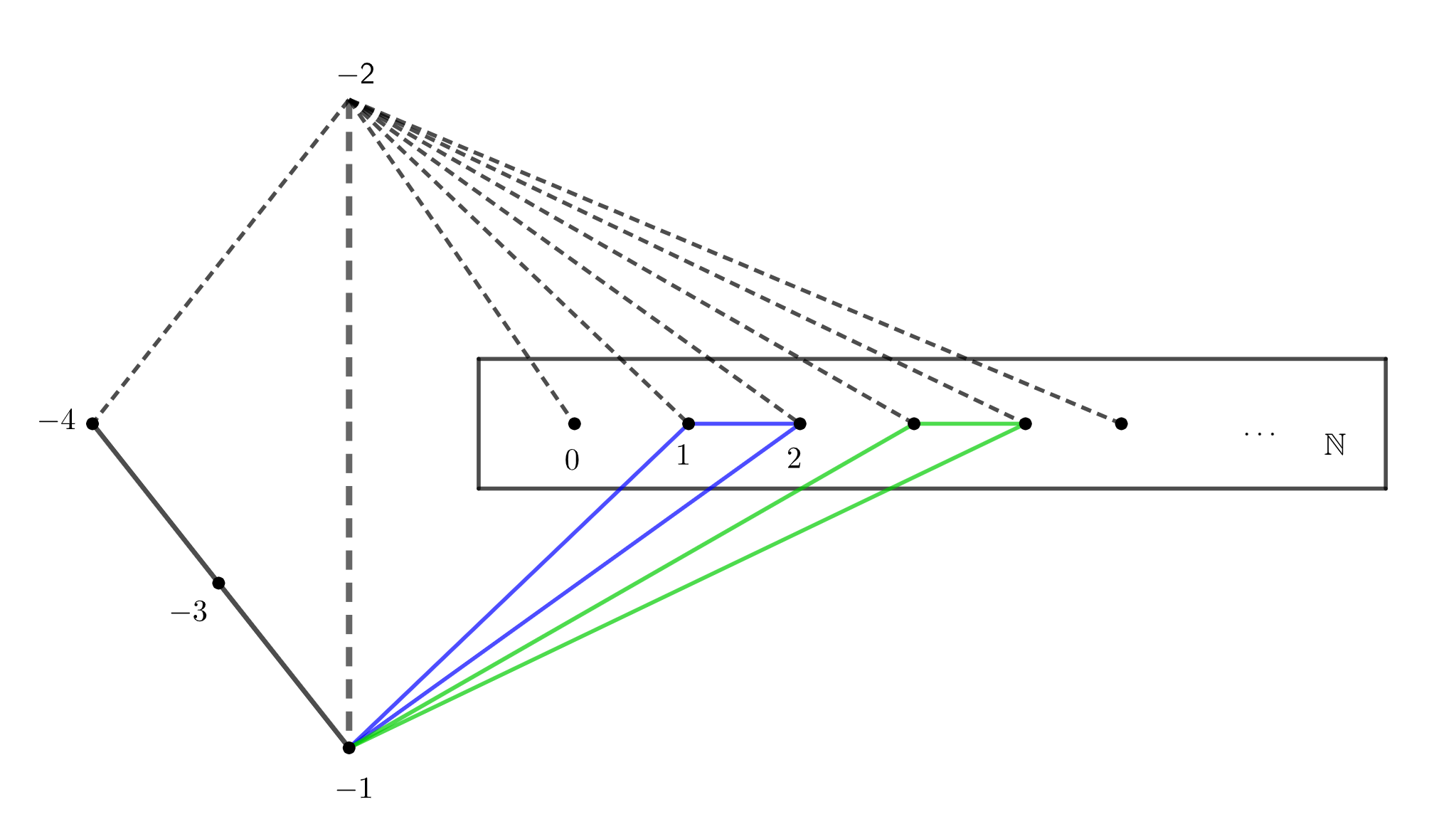}
		\caption{The construction of 3-uniform hypergraphs}
		\label{fig:3uh}
	\end{figure}
		
	Note that each of $\{-1, -2, -3, -4\}$ is definable in $F_0(M)$.
	The vertex $-1$ is the unique vertex that belongs to every edge, which is defined by $$\alpha_1(x) \coloneqq \forall y_1\forall y_2\forall y_3 E(y_1,y_2,y_3)\rightarrow \left(y_1=x\vee y_2=x\vee y_3=x  \right).$$
	The vertex $-2$ is the unique vertex that is not connected to exactly one other vertex, which is defined by
	$\alpha_2(x) \coloneqq \left( \forall y \neg E(x,y,\cdot) \right)^{(\cdot)=2}$.
	The vertex $-3$ is the unique vertex that is not connected to $-2$, which is defined by
	$$\alpha_3(x) \coloneqq \exists y\left(y\neq x\wedge\alpha_2(y)\wedge\forall z\neg E(x,y,z)\right).$$
	The vertex $-4$ is the unique vertex that is connected to $-1$ and $-3$, which is defined by
	$\alpha_4(x) \coloneqq \exists y\exists z \alpha_1(y)\wedge\alpha_3(z)\wedge E(x,y,z)$.
	Define $$\delta(x)\coloneqq\neg\left( \alpha_1(x)\vee\alpha_2(x)\vee\alpha_3(x)\vee\alpha_4(x) \right).$$
Then we have that
$ F_0(M)\models \delta(c)$ if and only if $c\in \N$. Define $\rho(x,y)$ as $\delta(x)\wedge\delta(y)\wedge\left(\exists z\alpha_1(z)\wedge E(x,y,z)\right)$. Then we have that $F_0(M)\models \rho(a,b)$ if and only if $M\models R(a,b)$. Thus, the graph $M$ is definable in the 3-uniform hypergraph $F_0(M)$. Let $\chi_1$ denote the following formula:
	\begin{align*} &\quad\alpha_1(\cdot)^{(\cdot)=1}\wedge\alpha_2(\cdot)^{(\cdot)=1}\wedge\alpha_3(\cdot)^{(\cdot)=1}\wedge\alpha_4(\cdot)^{(\cdot)=1}\\
	&\wedge\bigl( \forall x\forall y\forall z \left( E(x,y,z)\wedge\alpha_1(x)\wedge\delta(y)\right) \rightarrow \left( \delta(z)\vee\alpha_2(z) \right)  \bigr).
	\end{align*}
	Clearly, every countable model of $\Gamma_3\wedge\chi_1$ is a 3-uniform hypergraph that is isomorphic to some $F(M)$.
	
	Furthermore, we can use $L^E_{\omega_1,\omega}$-formulas in $F(M)$ to describe $L^R_{\omega_1,\omega}$-formulas in $M$. Recursively, for an $L^R_{\omega_1,\omega}$-formula $\varphi(\bar{x})$, we define its corresponding $L^E_{\omega_1,\omega}$-formula $\varphi^{F}(\bar{x})$ as follows:
\begin{itemize}
  \item $R(x,y)^{F}$ is $\rho(x,y)$
  \item $(x = y)^{F}$ is $x = y$
  \item $\left(\neg\varphi\right)^{F}(\bar{x})$  is $\neg\left( \varphi^{F}\right)(\bar{x})$
  \item $\left(\bigwedge\limits_{i\in\N}\varphi_i\right)^{F}(\bar{x})$ is $\bigwedge\limits_{i\in\N}\left(\varphi_i^{F}\right)$
  \item $\left(\bigvee\limits_{i\in\N}\varphi_i \right)^{F}(\bar{x})$ is $\bigvee\limits_{i\in\N}\left(\varphi_i^{F}\right)$
  \item $\left(\forall y\varphi(\bar{x},y)\right)^{F}$ is $\forall y\left(\delta(y)\rightarrow \varphi^{F}(\bar{x},y)\right)$
  \item $\left(\exists y\varphi(\bar{x},y)\right)^{F}$ is $\exists y\left(\delta(y)\wedge \varphi^{F}(\bar{x},y)\right)$
\end{itemize}
	
		Next, we will prove that $F\colon\mathrm{Mod}(\Gamma)\to\mathrm{Mod}(\Gamma_3)$ is a faithful Borel reduction.
	\begin{enumerate}[label=(\itshape\alph*)]
		\item Clearly, $F$ is continuous, and thus Borel.
		\item For all $M,N\in\mathrm{Mod}(\Gamma)$, we prove that $M\cong N \iff F(M)\cong F(N)$.
		\begin{enumerate}[label=(\itshape\roman*)]
			\item If $\varepsilon\colon M\to N$ is an isomorphism, then naturally
			\begin{displaymath}
				\bar{\varepsilon}(x) = \left\{
				\begin{array}{ll}
				x & {x\in\{-1,-2,-3,-4\}}\\
				\varepsilon(x) & {x\in\N}\\
				\end{array}
				\right.
			\end{displaymath}
			becomes an isomorphism from $F_0(M)$ to $F_0(N)$, and thus $F(M) \cong F(N)$.
			\item If $F(M) \cong F(N)$, then we have that $F_0(M) \cong F_0(N)$. Since $M$ and $N$ can be defined in $F_0(M)$ and $F_0(N)$ respectively, the restriction of an isomorphism from  $F_0(M)$ to $F_0(N)$ gives an isomorphism from $M$ to $N$.
		\end{enumerate}
		\item $F$ is faithful. Let $\hC = \mathrm{Mod}(\varphi)$ be an invariant Borel class in $\mathrm{Mod}(\Gamma)$, where $\varphi$ is an $L^R_{\omega_1,\omega}$-sentence. It is easy to verify that
		$$[F(\hC)]_{S_\infty} = [F(\mathrm{Mod}(\varphi))]_{S_\infty} = \mathrm{Mod}(\Gamma_3\wedge\varphi^F\wedge\chi_1).$$
		Hence, $[F(\hC)]_{S_\infty}$ is Borel, and thus $F$ is faithful.
		\item Since every added point is determined uniquely and they're fixed pointwise under automorphisms, it is easy to verify that $F$ preserves automorphisms, that is to say, $\Aut(M)\cong \Aut(F(M))$.
	\end{enumerate}
\end{proof}

\section{The class of countable partial Steiner triple systems is faithfully Borel complete}
Let $\Sigma_0$ be the theory of partial Steiner triple systems. In this section, we will prove the following theorem.

\begin{theorem}\label{pSTS}
The class of countable partial Steiner triple systems is faithfully Borel complete.
Moreover, there is a faithful Borel reduction $G\colon\Mod(\Gamma_3)\to\Mod(\Sigma_0)$ such that for every countable 3-uniform hypergraph $M$ with universe $\N$, we have $\Aut(G(M))\cong\Aut(M)$.
\end{theorem}
\begin{proof}
	By Proposition~\ref{hg}, the class of countable 3-uniform hypergraphs is Borel complete, we only need construct a faithful Borel reduction $G\colon\mathrm{Mod}(\Gamma_3)\to\mathrm{Mod}(\Sigma_0)$. Let $M\in\mathrm{Mod}(\Gamma_3)$ be a 3-uniform hypergraph. We will define $G(M)$ as follows:
	\begin{enumerate}[label=(\itshape\arabic*)]
		\item Fix a bijection $f\colon [\N]^3\to\N$.
	Then, for $i = 1,2,3$ we define $P_i\colon\N\to\N$ by $P_i(y) = x_i$, where $f(\{x_1,x_2,x_3\}) = y$ and $x_1>x_2>x_3$, and thus $P_1(y)>P_2(y)>P_3(y)$.
    Next, on $\N\cup \bigl( (\{ -9,-8,\cdots,-1 \} \cup\N) \times \N\bigr)$, we construct a partial Steiner triple system, denoted by $G_0(M)$, as follows:
		\begin{itemize}
			\item $E\left( (-1,m),(-2,m),(-3,m)\right)$  for all $m\in \N$
			\item $E\left( m,(-1,n), (m,n) \right)$  for all $m, n\in \N$
			\item $E\left( (P_1(n),n),(-4,n), (-5,n) \right)$ for all  $n\in \N$
            \item $E\left( (P_2(n),n),(-6,n), (-7,n) \right)$ for all  $n\in \N$
            \item $E\left( (P_3(n),n),(-8,n), (-9,n) \right)$ for all  $n\in \N$		
            \item $E((m,k),(n,k),(l,k))$ if $M \models E(m,n,l)$ and $k = f(\{m,n,l\})$
		\end{itemize}
See Figure~\ref{fig:psts} for the construction.
		\item Fix a bijection $g$ from $\N\cup \bigl( (\{-9,-8,\cdots,-1\} \cup\N) \times \N\bigr)$ to $\N$.
		Then $g$ turns $G_0(M)$ into a partial Steiner triple system on $\N$, denoted by $G(M)$.
	\end{enumerate}

The idea of the construction $G_0(M)$ is as follows: The set $\N$ consists of vertices in $M$, but edges in $M$ might not satisfy conditions for blocks. We assign different edges in $M$ to blocks in different rows in $\N^2$. Also, the elements of $\N$ correspond columns in $\N^2$. The set $\{-1\}\times\N$ is used to code different rows in $\N^2$, and the set $\{-2,-3\}\times\N$ is used to distinguish $\N$ and $\{-1\}\times\N$. The set $\{-4,-5,\cdots,-9\}\times\N$ is used to code three distinct points in the same row in $\N^2$, which might be in a block depending on whether there is an edge between the corresponding vertices in $M$. Note that the degree of each vertex in $\N$ and $\{-1\}\times\N$ is $\omega$, the degree of each vertex in $\{-2,-3\}\times\N$ and $\{-4,-5,\cdots,-9\}\times\N$ is 1, and the degree of each vertex in $\N^2$ is 1, 2, or 3.
	\begin{figure}[!htbp]
		\centering
		\includegraphics[width=0.80\textwidth]{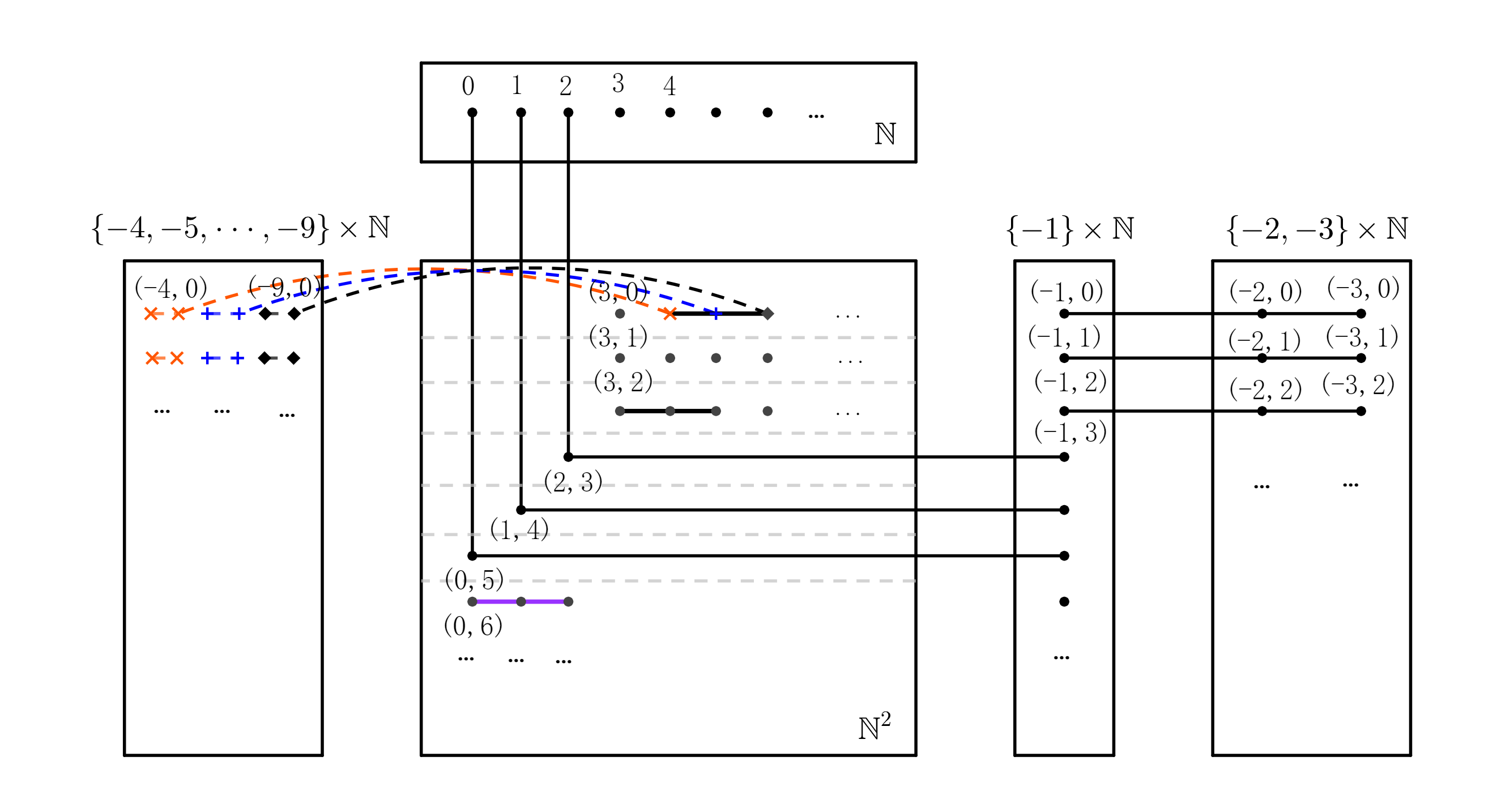}
		\caption{The construction of partial Steiner triple systems}
        \label{fig:psts}
	\end{figure}
		
	Note that $\left\lbrace -4, -5, -6, -7, -8, -9 \right\rbrace \times \N$, $\left\lbrace (P_i(n),n) \mid n\in\N, i=1,2,3 \right\rbrace$,  $\left\lbrace -2, -3 \right\rbrace \times \N$, $\left\lbrace -1 \right\rbrace \times \N$, $\N$, $\N^2$ are all definable in $G_0(M)$.
	The set $\left\lbrace -4, -5, -6, -7, -8, -9 \right\rbrace \times \N$ consists of degree 1 points that are connected to another degree 1 point and a point with degree $\leq 3$, which is defined by
	$$\tau_0(x) \coloneqq E(x,\cdot_1,\cdot_2)^{\{\cdot^2\}= 1} \wedge \left( \exists y\exists z E(x,y,z)\wedge E(y,\cdot_1,\cdot_2)^{\{\cdot^2\}= 1}\wedge E(z,\cdot_1,\cdot_2)^{\{\cdot^2\}\leqslant 3}\right).$$
	The set $\left\lbrace (P_i(n),n) \mid n\in\N, i=1,2,3 \right\rbrace$ consists of degree $>1$ points that are connected to two points in $\left\lbrace -4, -5, -6, -7, -8, -9 \right\rbrace \times \N$, which is defined by
	$$\tau(x) \coloneqq E(x,\cdot_1,\cdot_2)^{\{\cdot^2\}> 1} \wedge \left( \exists y\exists z \tau_0(y)\wedge\tau_0(z)\wedge E(x,y,z)  \right).$$
	The set $\left\lbrace -2, -3 \right\rbrace \times \N$ consists of degree 1 points that are connected to a degree 1 point and a degree $\omega$ point, which is defined by an $L^E_{\omega_1, \omega}$ formula:
	$$\theta_0(x) \coloneqq E(x,\cdot_1,\cdot_2)^{\{\cdot^2\}= 1} \wedge \left( \exists y\exists z E(x,y,z)\wedge E(y,\cdot_1,\cdot_2)^{\{\cdot^2\}= 1}\wedge E(z,\cdot_1,\cdot_2)^{\{\cdot^2\}= \omega} \right).$$
	The set $\left\lbrace -1 \right\rbrace \times \N$ consists of degree $\omega$ points that are connected to $\left\lbrace -2, -3 \right\rbrace \times \N$, which is defined by an $L^E_{\omega_1, \omega}$ formula:
	$$\theta(x) \coloneqq E(x,\cdot_1,\cdot_2)^{\{\cdot^2\}= \omega} \wedge \left( \exists y\exists z\theta_0(y)\wedge\theta_0(z)\wedge E(x,y,z) \right).$$
	The set $\N$ consists of the remaining degree $\omega$ points, which is defined by an $L^E_{\omega_1, \omega}$ formula:
	$$\delta(x) \coloneqq E(x,\cdot_1,\cdot_2)^{\{\cdot^2\}= \omega}\wedge\neg\theta(x).$$
	The set $\N^2$ consists of all remaining points, which is defined by
	$$\eta(x) \coloneqq \neg\left( \tau_0(x)\vee\theta_0(x)\vee\theta(x)\vee\delta(x) \right).$$
	Define $\rho(x,y,z)$ as
	\begin{align*}
	\delta(x)\wedge\delta(y)\wedge\delta(z)\wedge \bigl(&\exists w\exists x_0\exists y_0\exists z_0 \theta(w)\wedge\eta(x_0)\wedge\eta(y_0)\wedge\eta(z_0)\\
&\wedge E(x,w,x_0)\wedge E(y,w,y_0)\wedge E(z,w,z_0)\wedge E(x_0,y_0,z_0)\bigr).
	\end{align*}
	Then, it is easy to verify that $G_0(M)\models \rho(a,b,c)$ if and only if $M\models E(a,b,c)$, and thus, we have that the 3-uniform hypergraph $M$ is definable in the partial Steiner triple system $G_0(M)$.
	Also, we can define $G(M)$ by the following $L^E_{\omega_1, \omega}$-formula:
	\begin{align*}
	\chi_2 \coloneqq& \tau_0(\cdot)^{(\cdot)=\omega}\wedge\tau(\cdot)^{(\cdot)=\omega}\wedge\theta_0(\cdot)^{(\cdot)=\omega}\wedge\theta(\cdot)^{(\cdot)=\omega}\wedge\delta(\cdot)^{(\cdot)=\omega}\wedge\eta(\cdot)^{(\cdot)=\omega}\\
	&\wedge\bigl(\forall x\tau(x)\rightarrow\eta(x)\bigr)\\
    &\wedge\left( \forall x\theta_0(x) \rightarrow \left( \exists y E(x,y,\cdot)\wedge\theta(\cdot)\right)^{(\cdot)=1}  \right)\\
	&\wedge\left( \forall x\theta(x) \rightarrow \left( \exists y E(x,y,\cdot)\wedge\theta_0(\cdot)\right)^{(\cdot)=2}  \right)\\
	&\wedge\left( \forall x\tau_0(x) \rightarrow \left( \exists y E(x,y,\cdot)\wedge\tau(\cdot)\right)^{(\cdot)=1}  \right)\\
	&\wedge\left( \forall x\tau(x) \rightarrow \left( \exists y E(x,y,\cdot)\wedge\tau_0(\cdot)\right)^{(\cdot)=2}  \right)\\
	&\wedge\left( \forall x\eta(x) \rightarrow \left( \theta(\cdot_1)\wedge\delta(\cdot_2)\wedge E(x,\cdot_1,\cdot_2) \right)^{(\cdot^2)=1}  \right)\\
	&\wedge\left( \forall x\forall y \theta(x)\wedge\delta(y) \rightarrow \left( \eta(\cdot)\wedge E(x,y,\cdot) \right)^{(\cdot)=1}  \right)\\
	&\wedge\bigl( \forall x\forall y\forall z \delta(x)\wedge E(x,y,z) \rightarrow \left( \theta(y)\wedge\eta(z)\right) \vee\left( \eta(y)\wedge\theta(z)\right)  \bigr)\\
	&\wedge\bigl( \forall x\forall y\forall z (\theta(x)\wedge E(x,y,z)\wedge\neg\theta_0(y)) \rightarrow \left( \delta(y)\wedge\eta(z)\right) \vee\left( \eta(y)\wedge\delta(z)\right) \bigr)\\
	&\wedge\left( \forall x \tau(x) \rightarrow \left( E(x,\cdot_1,\cdot_2)^{\{\cdot^2\}= 2} \vee E(x,\cdot_1,\cdot_2)^{\{\cdot^2\}= 3} \right)\right) \\
	&\wedge\left( \forall x \neg\tau(x)\wedge\eta(x) \rightarrow   E(x,\cdot_1,\cdot_2)^{\{\cdot^2\}= 1} \right) \\
	&\wedge\left( \forall x \theta(x) \rightarrow \left( \exists y E(x,y,\cdot) \wedge \tau(\cdot) \right)^{(\cdot)=3} \right) \\
	&\wedge\left( \forall x_1\forall x_2\forall x_3\; \bigwedge_{i=1}^3\tau(x_i)\wedge E(x_1,x_2,x_3) \rightarrow \left( \exists y\exists z_1\exists z_2\exists z_3 \theta(y)\wedge\bigwedge_{i=1}^3 E(x_i,y,z_i) \right)\right)  \\
	&\wedge\left( \forall x_1\forall x_2\forall x_3 \bigwedge_{i=1}^3\delta(x_i) \rightarrow \left( \exists z_1\exists z_2\exists z_3 \theta(\cdot)\wedge \bigwedge_{i=1}^3\tau(z_i)\wedge\bigwedge_{i=1}^3 E(x_i,\cdot,z_i) \right)^{(\cdot)=1} \right).
	\end{align*}
	Clearly, every countable model of $\Sigma_0\wedge\chi_2$ is a countable partial Steiner triple system that is isomorphic to some $G(M)$.
	
	Then we can use $L^E_{\omega_1, \omega}$-formulas in $G(M)$ to describe $L^E_{\omega_1, \omega}$-formulas in $M$. Recursively, for an $L^E_{\omega_1, \omega}$-formula $\varphi(\bar{x})$, we define its corresponding $L^E_{\omega_1, \omega}$-formula $\varphi^G(\bar{x})$ as what we did for $\varphi^F(\bar{x})$ in the proof of Proposition \ref{hg}.

	Next, we will prove that $G\colon\mathrm{Mod}(\Gamma_3)\to\mathrm{Mod}(\Sigma_0)$ is a faithful Borel reduction.
	\begin{enumerate}[label=(\itshape\alph*)]
		\item Clearly, $G$ is Borel.
		\item For all $M,N\in\mathrm{Mod}(\Gamma_3)$, we prove that $M\cong N \iff G(M)\cong G(N)$.
		\begin{enumerate}[label=(\itshape\roman*)]
			\item If $\varepsilon\colon M\to N$ is an isomorphism, then
			\begin{displaymath}
			\bar{\varepsilon}(x) = \left\{
			\begin{array}{ll}
			\varepsilon(m)  & {x = m\in\N}\\
			(m,f(\varepsilon(a),\varepsilon(b),\varepsilon(c)))  & {x = (m,f(a,b,c))\in \{-1,-2,-3\}\times \N}\\
			(\varepsilon(m),f(\varepsilon(a),\varepsilon(b),\varepsilon(c)))  & {x = (m,f(a,b,c))\in \N^2}\\
			\end{array}
			\right.
			\end{displaymath}
			becomes a partial isomorphism from $G_0(M)$ to $G_0(N)$. Considering edges in $\tau_0(x)$ and $\tau(x)$, we can extend $\bar{\varepsilon}$ uniquely to an isomorphism from $G_0(M)$ to $G_0(N)$, and thus $G(M) \cong G(N)$.
			\item If $G(M) \cong G(N)$, then we have that $G_0(M) \cong G_0(N)$. Since $M$ and $N$ are defined by $\delta$ in $G_0(M)$ and $G_0(N)$ respectively, the restriction of an isomorphism from $G_0(M)$ to $G_0(N)$ gives an isomorphism from $M$ to $N$.
		\end{enumerate}
		\item $G$ is faithful. Let $\hC = \mathrm{Mod}(\varphi)$ be an invariant Borel class in $\mathrm{Mod}(\Gamma_3)$, where $\varphi$ is an $L^E_{\omega_1,\omega}$-sentence. It is easy to verify that
		$$[G(\hC)]_{S_\infty} = [G(\mathrm{Mod}(\varphi))]_{S_\infty} = \mathrm{Mod}(\Sigma_0\wedge\varphi^G\wedge\chi_2).$$
		Hence, $[G(\hC)]_{S_\infty}$ is Borel, and thus $G$ is faithful.
		\item In the construction of $G_0$, vertices in $M$ correspond columns in $G_0(M)$, and edges in $M$ correspond rows in $G_0(M)$.
Say $M=(\N,E)$ and take $\sigma\in\Aut(M)$. For distinct $m,n,l\in\N$, define $\Phi(m)=\sigma(m)$, and $$\Phi(i,f(\{m,n,l\}))=(i,f(\{\sigma(m),\sigma(n),\sigma(l)\})),$$ where $i\in\{-9,-8,\cdots,-1\}\cup\N$. Based on the construction of $G_0(M)$, it is easy to verify that $\Phi\in\Aut(G_0(M))$. Thus, every automorphism of the hypergraph $M$ can be uniquely extended to an automorphism of the partial Steiner triple system $G_0(M)$. On the other hand, the restriction of every automorphism of $G_0(M)$ on the hypergraph $\delta(G_0(M))$ is also an automorphism. Hence, $G_0$ preserves automorphisms, that is to say, $$\Aut(M) \cong \Aut(G_0(M))\cong\Aut(G(M)).$$
	\end{enumerate}
\end{proof}

\section{The class of countable Steiner triple systems is faithfully Borel complete}
Let $\Sigma$ be the theory of Steiner triple systems. In this section, we will prove the Main Theorem. \\
\emph{Proof of Main Theorem.}
In Theorem~\ref{pSTS}, we build a faithful Borel reduction $G\colon\mathrm{Mod}(\Gamma_3)\to\mathrm{Mod}(\Sigma_0)$. Actually, we proved that the class of $\mathrm{Mod}(\Sigma_0\wedge\chi_2)$ is faithfully Borel complete. Thus, to prove that the class of countable Steiner triple systems is Borel complete, we only need construct a faithful Borel reduction $H\colon\mathrm{Mod}(\Sigma_0\wedge\chi_2)\to\mathrm{Mod}(\Sigma)$.
Let $M\in\mathrm{Mod}(\Sigma_0\wedge\chi_2)$ be a countable partial Steiner triple system of the form shown in Figure~\ref{fig:psts}. We will define $H(M)$ in the following 4 steps.

\textbf{Step 1:} First, we introduce the simplest nontrivial Steiner triple system whose automorphism group is $S_3$. It turns out to be a Steiner triple system with 15 points and 35 blocks. In \cite[II.1.2, 1.28]{CD}, all 80 non isomorphic Steiner triple systems of order 15 are constructed. Let $T$ denote the $\# 43$ in the list. Then by \cite[II.1.2, 1.29]{CD}, $T$ contains no Fano planes and thus no nontrivial subsystems, and the automorphism group of $T$ is of order 6. To be precise, $T=(V_T,E_T)$, where $V_T=\{1,2,\cdots,14,15\}$ and
\begin{align*}
\qquad E_T=\big\{&\{1,2,3\}, \{1,4,5\}, \{1,6,7\}, \{1,8,9\}, \{1,10,11\}, \{1,12,13\}, \{1,14,15\}, \\
&\{2,4,6\}, \{2,5,7\}, \{2,8,10\}, \{2,9,11\}, \{2,12,14\},\{2,13,15\}, \{3,4,8\}, \\
&\{3,5,9\}, \{3,6,12\}, \{3,7,14\}, \{3,10,15\}, \{3,11,13\}, \{4,7,10\}, \{4,9,15\}, \\
&\{4,11,12\}, \{4,13,14\}, \{5,6,11\}, \{5,8,13\}, \{5,10,14\}, \{5,12,15\}, \\
&\{6,8,15\}, \{6,9,14\}, \{6,10,13\}, \{7,8,12\}, \{7,9,13\},\{7,11,15\}, \\
&\{8,11,14\}, \{9,10,12\}\big\}.
\end{align*}
A computation in Magma shows that 
\begin{align*}
\Aut(T)=&\big\{\mathrm{Id},\\
    &\alpha=(1, 8)(2, 13)(3, 5)(6, 14)(7, 11)(10, 12),\\
    &\beta\alpha=(1, 9)(2, 15)(3, 4)(6, 10)(7, 12)(11, 14),\\
    &\beta^2\alpha=(4, 5)(6, 7)(8, 9)(10, 11)(12, 14)(13, 15),\\
    &\beta=(1, 8, 9)(2, 13, 15)(3, 5, 4)(6, 11, 12)(7, 14, 10),\\
    &\beta^2=(1, 9, 8)(2, 15, 13)(3, 4, 5)(6, 12, 11)(7, 10, 14)
\big\}\cong S_3 .
\end{align*}
Note that in $\Aut(T)$, order 2 automorphisms have fixed points, while order 3 automorphisms have no fixed points. Precisely, $\{4,9,15\}$ is fixed by $\alpha$ pointwise, $\{5,8,13\}$ is fixed by $\beta\alpha$ pointwise, and $\{1,2,3\}$ is fixed by $\beta^2\alpha$ pointwise. Also, every iteration orbit of order 3 automorphisms is of size 3, and there are 5 iteration orbits. Among those 5 iteration orbits, there are only two orbits are blocks which are $\{1, 8, 9\}$ and $\{2, 13, 15\}$. Among the non-block iteration orbits, $\{3,4,5\}$ is the unique one where every point is fixed by some automorphism, and $3*4=8$, $3*5=9$, $4*5=1$. From the above discussion, we can identify the block $\{2,13,15\}$ from the Steiner triple system $T$, but we are unable to identify the point $\{2\}$ from the block $\{2,13,15\}$. Indeed, given $\sigma\in\Aut(T)$, the restriction of $\sigma$ on $\{2,13,15\}$ is a permutation. Also, every permutation of $\{2,13,15\}$ can be extended to an automorphism of $T$.

\textbf{Step 2:} For each block $\{x,y,z\}$ of $M$, we attach a Steiner triple system $T$ such that $\{x,y,z\}$ corresponds $\{2,13,15\}$. From the construction of $\{2,13,15\}$ in $T$, each permutation works. Then we get a partial Steiner triple system $M^T$.

\textbf{Step 3:} Let $N$ denote the Steiner triple system freely generated by $M^T$.
		

\textbf{Step 4:} After enumeration, we make $N$ a Steiner triple system on $\N$, denoted by $H(M)$.

Note that $M\in\mathrm{Mod}(\Sigma_0\wedge\chi_2)$ is isomorphic to $G(A)$ for some 3-uniform hypergraph $A$, and based on the construction of $G(A)$, we have that $M$ contains no subsystem isomorphic to $T$. Also, free generation does not generate any new system isomorphic to $T$. Hence, $M$ and $M^T$ can be identified in $N$ using $T$. 

Since the Steiner triple system $T$ is finite, we let $\tau(x_1,\cdots,x_{15})$ denote the $L^E$-formula defining $T$ in a way such that $x_i$ corresponds the point $\{i\}$ for each $1\leq i \leq 15$. 
Let $\xi(x_1,\cdots, x_{15})$ denote $\bigvee\limits_{\pi\in S_{15}}\tau\,(x_{\pi(1)},\cdots, x_{\pi(15)})$. 
Then, for every partial Steiner triple system $S$,
we have that $S\models \xi(c_1,c_2,\cdots,c_{15})$ if and only if $S|_{\left\lbrace c_1,c_2,\cdots,c_{15}\right\rbrace } \cong T$, where $c_1,c_2,\cdots,c_{15}\in S$.
Let $\xi_3(x_1, x_2, x_3)$ denote $\exists x_{4}\cdots\exists x_{15} \xi(x_1, \cdots, x_{15})$. Then $\xi_3$ expresses that the 3 points $\{x_1,x_2,x_3\}$ are in a system isomorphic to $T$. 
Define $\tau_{2,13,15}(x,y,z)$ as
$$\exists x_1\exists x_3\cdots\exists x_{12}\exists x_{14}\tau(x_1,x,x_3,\cdots,x_{12},y,x_{14},z).$$
Then, for every partial Steiner triple system $S$ and for every $a,b,c\in S$,
we have that $S\models \tau_{2,13,15}(a,b,c)$ if and only if $S$ contains a subsystem $T$ such that $\{a,b,c\}$ correspond $\{2,13,15\}$ in $T$.
	Note that blocks in $M^T$ are exactly blocks in subsystems isomorphic to $T$ in $N$. Define $\widetilde{\rho}(x,y,z)$ as $E(x,y,z)\wedge\xi_3(x,y,z)$.
	Then, $N\models \widetilde{\rho}(a, b, c)$ if and only if $M^T\models E(a, b, c)$. Since all points in $M^T$ are in some blocks, $M^T$ can be determined by blocks. Define $\widetilde{\delta}(x)$ as $\exists y\exists z\widetilde{\rho}(x,y,z)$. Then, $N\models \widetilde{\delta}(c)$ if and only if  $c\in M^T$. Consider systems $T$ that are contained in $N$. Note that points $\{2,13,15\}$ in $T$ correspond points in $M$. Define $\delta(x)$ as $\exists y\exists z \tau_{2,13,15}(x,y,z)$. Then, $N\models \delta(c)$ if and only if $c\in M$.
	Note that blocks in $M$ correspond blocks $\{2,13,15\}$ in subsystems $T$ of $N$.
	Thus, $N\models \tau_{2,13,15}(a, b, c)$ if and only if $M\models E(a, b, c)$. Then, $\left(\delta(\cdot), \tau_{2,13,15}(\cdot_1,\cdot_2,\cdot_3) \right)$ interprets $\left(M, E(\cdot_1,\cdot_2,\cdot_3)\right)$ in $N$.

	Further, we can use $L^E_{\omega_1,\omega}$-formulas in $H(M)$ to describe $L^E_{\omega_1,\omega}$-formulas in $M$. Recursively, for an $L^E_{\omega_1,\omega}$-formula $\varphi(\bar{x})$, we define its corresponding $L^E_{\omega_1,\omega}$-formula $\varphi^H(\bar{x})$ as what we did for $\varphi^F(\bar{x})$ in the proof of Proposition \ref{hg}.	
To express freely generated $N$, define $L^E_{\omega_1,\omega}$-formula as follows:
	\begin{align*}
	\lambda_0(x) \coloneqq & \widetilde{\delta}(x)\\
	\lambda_{n+1}(x) \coloneqq & \exists y\exists z\lambda_n(y)\wedge\lambda_n(z)\wedge E(x,y,z), \forall n\in\N\\
	\lambda'_{n+1}(x) \coloneqq & \neg\lambda_{n}(x)\wedge\lambda_{n+1}(x), \forall n\in\N\\
	\mu_{n}(x,y,z) \coloneqq &  E(x,y,z)\wedge\lambda_{n}(x)\wedge\lambda_{n}(y)\wedge\lambda_{n}(z), \forall n\in\N\\
	\mu'_{n+1}(x,y,z) \coloneqq & \neg\mu_{n}(x,y,z)\wedge\mu_{n+1}(x,y,z), \forall n\in\N.
	\end{align*}
Roughly speaking, $\lambda_0$ gives $M^T$, which is considered as Level 0, and $\lambda_0\subseteq\lambda_1\subseteq\cdots$. For each $n\geq 1$, $\lambda_n$ is considered as Level $n$. 
Further, define
$$\widetilde{\tau}(u,v,w,y_1,\cdots,y_{12})\coloneqq  \tau(y_1,u,y_2,\cdots,y_{11},v,y_{12},w),$$
and for each $1\leq n\leq 12$, define
$$\widetilde{\tau_n}(u,v,w,y)\coloneqq \exists y_1\cdots\exists y_{n-1}\exists y_{n+1}\cdots\exists y_{12}\widetilde{\tau}(u,v,w,y_1,\cdots,y_{n-1},y,y_{n+1},\cdots,y_{12}).$$

Now, we can define $H(M)$ by the following $L^E_{\omega_1,\omega}$-formula:
	\begin{align*}
	\chi_3 \coloneqq& \forall x\bigl(\neg\delta(x)\wedge\widetilde{\delta}(x)\bigr)\rightarrow \bigvee\limits_{n=1}^{12}\left( \widetilde{\tau_n}(\cdot_1,\cdot_2,\cdot_3,x)\right)^{\{\cdot^3\} = 1 }\\
    &\wedge\forall x\forall y\forall z\tau_{2,13,15}(x,y,z)\rightarrow \left(\bigvee\limits_{n=1}^{12}\widetilde{\tau_n}(x,y,z,\cdot)\right)^{(\cdot)=12}\\
	&\wedge\forall x\forall y\forall z\mu_0(x,y,z)\leftrightarrow \xi_3(x,y,z)\\
	&\wedge\bigwedge\limits_{i\in\omega}\bigl( \forall x\forall y\lambda_i(x)\wedge\lambda_i(y) \rightarrow \exists z\mu_{i+1}(x,y,z) \bigr)\\
	&\wedge\bigwedge\limits_{i\in\omega}\left( \forall x\lambda^{\prime}_{i+1}(x) \rightarrow \left( \mu_{i+1}(x,\cdot_1,\cdot_2)\wedge\lambda_{i}(\cdot_1)\wedge\lambda_{i}(\cdot_2)\right) ^{\{\cdot^2\}=1}   \right)\\	
	&\wedge\bigwedge\limits_{i\in\omega}\left( \forall x\lambda^{\prime}_{i+1}(x) \rightarrow  \left( \exists y \mu_{i+1}(x,y,\cdot)\wedge\lambda_{i+1}(\cdot) \right)^{(\cdot)=2}  \right)\\
	&\wedge \forall x \bigvee\limits_{i\in\omega} \lambda_i(x).
	\end{align*}

	Clearly, for every countable Steiner triple system $N\in \mathrm{Mod}(\chi_3\wedge\Sigma\wedge\chi_2^{H})$, there is a countable partial Steiner triple system $M\in \mathrm{Mod}(\Sigma_0\wedge\chi_2)$ such that $N \cong H(M)$.
	
	Next, we will prove that $H\colon\mathrm{Mod}(\Sigma_0\wedge\chi_2)\to\mathrm{Mod}(\Sigma)$ is a faithful Borel reduction.
	\begin{enumerate}[label=(\itshape\alph*)]
		\item Clearly, $H$ is Borel.
		\item For all $M, M'\in\mathrm{Mod}(\Sigma_0\wedge\chi_2)$, let $N$ and $N'$ denote the Steiner triple systems generated by $M$ and $M'$ respectively, as in the previous construction.
		\begin{enumerate}[label=(\itshape\roman*)]
			\item If $\varepsilon\colon M\to M'$ is an isomorphism, then $N \cong N'$, and thus, $H(M) \cong H(M')$.
			\item If $H(M) \cong H(M')$, then there is an isomorphism $\pi\colon N \to N'$ so that $\pi|_{\delta(\cdot)}\colon M\to M'$ gives an isomorphism from $M$ to $M'$.
        \end{enumerate}
		\item $H$ is faithful. Let $\hC = \mathrm{Mod}(\varphi)$ be an invariant Borel class in $\mathrm{Mod}(\Sigma_0\wedge\chi_2)$, where $\varphi$ is an $L^E_{\omega_1,\omega}$-sentence. It is easy to verify that
		$$[H(\hC)]_{S_\infty} = [H(\mathrm{Mod}(\varphi))]_{S_\infty} = \mathrm{Mod}(\varphi^{H}\wedge\chi_3\wedge\Sigma\wedge\chi_2^{H}).$$
		Hence, $[H(\hC)]_{S_\infty}$ is Borel, and thus $H$ is faithful.
		
		\item In Step 2 of the construction of $H$, for every block of $M$ we attach a Steiner triple system $T$ to get $M^T$. 

\flushleft\textbf{Claim:} $\Aut(M)\cong\Aut(M^T)$.
\flushleft\emph{Proof of the Claim.}
Note that $\Aut(T) = S_3$, and the restriction of every automorphism of $T$ on $\{2,13,15\}$ is a permutation. Also, every permutation of $\{2,13,15\}$ in $T$ can be extended to an automorphism of $T$. For each $\sigma\in\Aut(M)$, we extend $\sigma$ to $\sigma^T\in\Aut(M^T)$ in the following way. Let $B$ be a block of $M$. From the construction of $M^T$, there is a Steiner triple system $S\subseteq M^T$ and an isomorphism $\alpha\colon T\to S$ such that $\{\alpha(2),\alpha(13),\alpha(15)\}=B$. Also, there is an isomorphism $\beta\colon T\to \sigma(S)$ such that $\{\beta(2),\beta(13),\beta(15)\}=\sigma(B)$. Then, there is a permutation $\pi$ of $\{2,13,15\}$ such that $\sigma(\alpha(2))=\beta(\pi(2))$, $\sigma(\alpha(13))=\beta(\pi(13))$, $\sigma(\alpha(15))=\beta(\pi(15))$. The permutation $\pi$ can be extended to an automorphism $\hat\pi$ of $T$. Define $\sigma^T(\alpha(i))\coloneqq \beta(\hat\pi(i))$ for $1\leq i\leq 15$. Since $B$ is arbitrary, $\sigma^T$ is defined on $M^T$. It is easy to verify that (i) $\sigma^T\in\Aut(M^T)$, (ii) $(\sigma_1\sigma_2)^T=\sigma_1^T\sigma_2^T$ for $\sigma_1,\sigma_2\in\Aut(M)$, and (iii) $\Id_M^T=\Id_{M^T}$. Define $f\colon\Aut(M)\to\Aut(M^T)$ by $\sigma\mapsto\sigma^T$. Then, $f$ is a homomorphism. On the other hand, take $\Phi\in\Aut(M^T)$. Clearly, $\Phi$ maps subsystems that are isomorphic to $T$ to subsystems isomorphic to $T$. Because $\{2,13,15\}$ can be identified from $T$, we have that for every block $B$ of $M$, $\Phi(B)$ is also a block of $M$. Since every point in $M$ is in a block of $M$, we have that $\Phi(M)=M$. It is easy to verify that (i) $\Phi|_M\in\Aut(M)$, (ii) $(\Phi_1\Phi_2)|_M=\Phi_1|_M\Phi_2|_M$ for $\Phi_1,\Phi_2\in\Aut(M^T)$, and (iii) $\Id_{M^T}|_M=\Id_M$. Define $g\colon\Aut(M^T)\to\Aut(M)$ by $\Phi\mapsto\Phi|_M$. Then, $g$ is a homomorphism. Note that $f\circ g=\Id$ and $g\circ f=\Id$. Hence, $\Aut(M)\cong\Aut(M^T)$. \hfill $\Box_{Claim}$

Note that $N$ is freely generated by $M^T$. During the free generation construction, no finite Steiner triple subsystems are generated, and thus all subsystems that are isomorphic to $T$ are in $M^T$. Let $\tau\in\Aut(N)$. Since every point in $M^T$ is in a subsystem isomorphic to $T$, we have that $\tau(M^T)=M^T$. Then by Lemma \ref{free generation preserves Aut}, $\Aut(M^T)\cong\Aut(N)$, and thus, $\Aut(M)\cong\Aut(H(M))$.
	\end{enumerate}

By Proposition~\ref{hg} and Theorem~\ref{pSTS}, $\theta=H\circ G\circ F$ is as desired.
\qed

\subsection*{Acknowledgements}
We would like to thank Tao Feng for helpful discussions related to this project and for assisting us in using Magma. This work grew out of Guangyin Ma's 2024 master's thesis at Beijing Jiaotong University, supervised by Shichang Song.
This work was partially supported by ``National Natural Science Fund of China" with Grant No.\,12271023, and by the Fundamental Research Funds for the Central Universities, Beijing Jiaotong University, Grant No.\,2025JBZX010.

AI models were never used in this project.


\begin{thebibliography}{99}




\normalsize
\baselineskip=17pt


\bibitem{BC} S.\,Barbina, E.\,Casanovas, \emph{Model theory of Steiner triple systems}, J. Math. Log. 20 (2020) 2050010.
\bibitem{CG} R.\,Camerlo, S.\,Gao, \emph{The completeness of the isomorphism relation for countable Boolean algebras}, Trans. Amer. Math. Soc. 353 (2011) 491--518.
\bibitem{CCP} J.\,Clemens, S.\,Coskey, S.\,Potter, \emph{On the classification of vertex-transitive structures}, Arch. Math. Log. 58 (2019) 565--574.
\bibitem{CR} C.\,Colbourn, A.\,Rosa, \emph{Triple Systems}. Oxford: Oxford University Press, 1999.
\bibitem{CD} C.\,Colbourn, J.\,Dinitz, \emph{Handbook of Combinatorial Designs}, Second Edition, New York: Chapman and Hall/CRC, 2006.
\bibitem{FS} H.\,Friedman, L.\,Stanley, \emph{A Borel reducibility theory for classes of countable structures}, J. Symb. Log. 54 (1989) 894--914.
\bibitem{G} S.\,Gao, \emph{Invariant Descriptive Set Theory}, Pure and Applied Mathematics, vol. 293, Boca Raton: Taylor \& Francis Group, 2009.
\bibitem{HW} D.\,Horsley, B.\,S.\,Webb, \emph{Countable homogeneous Steiner triple systems avoiding specified subsystems}, J. Comb. Theory, Ser. A 180 (2021) 105434.
\bibitem{LU} M.\,C.\,Laskowski, D.\,S.\,Ulrich, \emph{Borel complexity of modules}, Ann. Pure Appl. Logic, 177 (2026) 103780.
\bibitem{PS} G.\,Paolini, S.\,Shelah, \emph{Torsion free abelian groups are Borel complete}, Ann. of Math. (2) 199 (2024) 1177--1224.
\bibitem{PS2} G.\,Paolini, S.\,Shelah, \emph{Torsion-free abelian groups are faithfully Borel complete and pure embeddability is a complete analytic quasi-order}, Sci. China Math. 68 (2025) 2809--2814.

\end{thebibliography}
\end{document}